\documentclass[letterpaper, 10 pt, conference]{ieeeconf}  %

\IEEEoverridecommandlockouts                              %

\makeatletter
\let\NAT@parse\undefined
\makeatother

\usepackage{mystyle}
\usepackage{list_macro}

\usepackage{mathrsfs}  
\usepackage{nicematrix}
\usepackage{scenariotree}
\usepackage{multirow}
\usepackage{cite} 
\usepackage{relsize}
\usepackage[para]{threeparttable}

\newcommand*{\md}{\xi}

\showbothfalse

\title{\LARGE \bf
Risk-Sensitive Model Predictive Control for Interaction-Aware Planning\\--- A Sequential Convexification Algorithm ---
}

\author{Renzi Wang \and Mathijs Schuurmans \and Panagiotis Patrinos%
\thanks{
        Work supported by
    the Research Foundation Flanders (FWO) postdoctoral grant 12Y7622N and research projects
    G081222N, G033822N, and G0A0920N; 
    Research Council KU Leuven C1 project No. C14/24/103; 
    European Union’s Horizon 2020 research and innovation programme under the Marie Skłodowska-Curie grant agreement No. 953348.
}%
\thanks{KU Leuven, Department of Electrical Engineering \textsc{esat-stadius} -- %
Kasteelpark Arenberg 10, bus 2446, B-3001 Leuven, Belgium
\newline
{\sf
        \{%
                \href{mailto:renzi.wang@kuleuven.be}{renzi.wang},
                \href{mailto:mathijs.schuurmans@kuleuven.be}{mathijs.schuurmans},
                \href{mailto:panos.patrinos@kuleuven.be}{panos.patrinos}%
        \}%
        \href{mailto:renzi.wang@kuleuven.be,mathijs.schuurmans@kuleuven.be,panos.patrinos@kuleuven.be}{@kuleuven.be}%
}
}%
}

\begin{document}

\maketitle
\thispagestyle{empty}
\pagestyle{empty}

\begin{abstract}
    This paper considers risk-sensitive model predictive control for stochastic systems with a decision-dependent distribution.
    This class of systems is commonly found in human-robot interaction scenarios.
    We derive computationally tractable convex upper bounds to both the objective function, 
    and to frequently used penalty terms for collision avoidance,
    allowing us to efficiently solve the generally nonconvex optimal control problem 
    as a sequence of convex problems. 
    Simulations of a robot navigating a corridor demonstrate the effectiveness
    and the computational advantage of the proposed approach.
\end{abstract}

\section{Introduction}
Stochastic dynamical systems are \rev{
	used for a wide variety of control tasks.
In particular, for applications such as autonomous driving and robot navigation,
where it is crucial to predict the behavior of surrounding humans, which cannot
be assumed to be perfectly deterministic\cite{montali2023waymo}.
}{widely used in control tasks,
particularly in applications such as autonomous driving and robot navigation,
where predicting non-deterministic human behavior is crucial \cite{montali2023waymo}.}
\rev{Traditionally, stochastic influences were typically modeled as exogenous
disturbances to the controlled system.
More recently, however, the use of}{While stochastic influences were traditionally modeled as exogenous disturbances,} mixture models has
gained popularity for this purpose \cite{wang2024imitation, salzmann2020trajectron++}.
In contrast to the exogenous stochastic switching models such as Markov jump models \cite{costa2005discrete},
this class of models introduce a dependency of the probability on the system's state,
significantly boosting its ability to model interaction between the controlled system
and \rev{the uncertainty in the environment}{environmental uncertainty}.

Conventional stochastic model predictive control (MPC) methods optimize the expected value of
the cost function (known as \emph{risk-neutral} cost).
This typically leads to tractable formulations, as the linearity of the
expectation preserves convexity of the cost function for exogenous uncertainty.
However, this is no longer the case for decision-dependent distributions, which
generally lead to non-convex problems, even when the underlying cost function is
convex.
Furthermore, studies in human behavior modeling  \cite{majumdar2017risk, ratliff2019inverse} show that humans are risk-sensitive.
This insight suggests that designing risk-sensitive controller would enable more natural human-robot interaction.
Based on those two observations, we aim to develop a risk-sensitive optimal control
algorithm using the exponential utility function.
Besides being beneficial from a modeling point of view, we observe that
the exponential utility function introduces additional structure that facilitates
the development of a tailored solution method\rev{ that}{. This method} is specifically designed to address
the nonconvexity arising from the state-dependent distributions.

\textbf{Contributions}
We summarize the contributions of this paper as follows:
\begin{inlinelist}
	\item We propose a risk-sensitive MPC formulation for a \rev[contribution]{mixture model
	with state-dependent distributions}{mixture-of-experts (MoE) model \cite{jordan1994hierarchical} that expands beyond traditional Gaussian-only uncertainty models while explicitly accounting for the state-dependency—a crucial property for modeling interactive behaviors.}
	This formulation can be interpreted as a smoothed version of a two-player game, which may be either
	cooperative or adversarial, based on the value of a hyperparameter;
	\item We develop an algorithm based on Majorization-Minimization (MM) principle \cite{lange2016mm} 
	that \rev{}{explicitly exploits the structure of the proposed risk-sensitive formulation. The algorithm} converts the nonconvex optimization problem into a sequence of convex problems.
	This includes a novel reformulation for the nonconvex collision penalty, which is a critical component in navigation systems.
	\item Our numerical experiments validate the effectiveness of both our algorithm and the formulation,
	demonstrate how the risk-sensitivity parameter impacts interaction behavior,
	and show significant computational advantages over off-the-shelf solvers for nonlinear programming.
\end{inlinelist}

\subsection{Related Work}\label{sec: related_work}
\subsubsection{Controlling a stochastic mixture model}

Human behavior modeling in interactive settings has increasingly employed mixture models to capture its multi-modal nature.
When formulating optimal control problems,
scenario trees are commonly employed to roll out the dynamics.
Several works \cite{wang2023interaction, hu2022active, schmerling2018multimodal} focus on minimizing the expected value of the cost function,
with \cite{wang2023interaction, hu2022active} employing general nonlinear programming solvers and \cite{schmerling2018multimodal} utilizing sampling-based methods.
\cite{chen2022interactive} considers both expectation and Conditional Value-at-Risk in their objective functions,
and solves the resulting problems via sequential quadratic programming.

\subsubsection{Risk-sensitive control using the exponential utility function}
The risk sensitive control problem with the exponential utility function has a long history \cite{howard1972risk, jacobson1973optimal,whittle1989entropy,whittle2002risk}.
In recent years, this type of method has also been applied in reinforcement learning \cite{moos2022robust, noorani2023risk, liang2024bridging}.
When applied to interactive settings,
\cite{medina2013risk,wang2020game} consider the joint dynamics to be linear with an additive Gaussian noise and solve the problem based on Ricatti recursion.
\cite{nishimura2020risk} employs a probabilistic machine learning model \cite{salzmann2020trajectron++} that can be conditioned on future trajectories,
and solves the optimal control problem using an approach based on sensitivity analysis.
However, their sensitivity analysis does not account for how perturbed future trajectories affect the distribution, and consequently the cost.
In contrast, our work explicitly considers the decision-dependent distribution in the mixture model.

\subsubsection{Majorization-minimization-based solver for control problems}
The expectation-maximization (EM) \cite{dempster1977maximum} algorithm has already been used to solve optimal control problems \cite{peters2010relative, levine2018reinforcement, noorani2023risk},
where Jensen's inequality is used to derive the majorizer.
Our work extends this beyond Jensen's inequality.
Unlike \cite{moehle2021risk}, which applies the MM principle only as a substep to find the optimal disturbance sequence in linear systems, 
we apply MM to the entire optimal control problem for mixture models.

The paper is organized as follows:
We first present the risk-sensitive formulation and its interpretation in \cref{sec: problem_formulation}.
We then derive the solution method in \cref{sec: mm_solver}.
Finally, we propose an addition formulation for collision penalty and validate the proposed method through a numerical experiment in \cref{sec: numerical_experiment}.

\rev{
\subsection*{Notation}}{}
\rev{}{\textbf{Notation}}
Let $\R^d_+$ denote the $d$-dimensional positive orthant. 
Let $\mathbf{1}_d = \bsmat{1 & \dots & 1} \in \re^d$.
We denote the cone of $m \times m$ positive definite matrices by $\mathbb{S}_{++}^{m}$,
and the probability simplex $\Delta^d = \{p \in \re^d_{+} \mid \mathbf{1}_d^\top p = 1\}$.
Define $\lse: \R^{d} \to \R$ as
\(
    \lse(x) \! =\! \ln \big(
        \textstyle{\sum_{j=1}^{d}} \exp(x_j) 
    \big),
\)
and the softmax function $\sigma: \R^d \to \Delta_{d}$ 
with components 
\(
    \sigma_i(x) \! = \!\exp\big(x_i\big) / \sum_{j=1}^{d}\exp\big(x_j\big).
\)
\rev[notation_expectation]{
We use $\E_\Pi[x]$ to represent the expectation of random variable $x$ under the distribution $\Pi$,
and $\var_\Pi[x]$ to denote the variance of $x$ under $\Pi$.}{
    For discrete distribution $\Pi \in \Delta^n$ and 
    $X \in \re^n$ representing realizations of random variable $x$,
    we denote $\E_\Pi[X] = \sum_{i = 1}^n \Pi_i X_i$
    and $\var_\Pi[X] = \sum_{i = 1}^n \Pi_i (X_i - \E_\Pi[X])^2$.
}
\rev[notation_N]{}{Let $\N_{[a, b]} = \N \cap [a, b]$.}
\rev[notation_vecmax]{}{We denote by $\vecmin{L}$ ($\vecmax{L}$), the smallest (largest) elements of a vector $L$.}

\section{Problem Formulation}\label{sec: problem_formulation}

We consider a class of stochastic systems \rev{of}{in} the form
\rev{}{of an MoE model \cite{jordan1994hierarchical}:}
\begin{subequations}\label{eq: system}
    \begin{align}
        \xi_{t} \sim &\; \sigma(\rev{\Theta^\top x_t}{\Theta x_t}), \label{eq: switching_mechanism}\\
        x_{t+1} = &\; A_{\xi_t} x_t + B_{\xi_t} u_t,
    \end{align}
\end{subequations}
where for all $t \geq 0$, 
$x_t \in \mathbb{X} \subset \re^{n_x}$ denotes the continuous system state,
and $u_t \in \mathbb{U} \subset \re^{n_u}$ is the control input. 
The sets $\mathbb{X}$ and $\mathbb{U}$ are nonempty closed convex sets.
The state $x_t$ is governed by $d$ linear \rev{systems that switch}{experts that selected} 
according to the value of the random variable $\xi_t \in \Xi = \{1, \dots, d\}$.
We refer to $\xi_t$ as the \textit{mode} at time $t$. 
At every time step $t$, $\xi_t$ \rev{get}{is} sampled randomly according to the \rev{}{gate} distribution \eqref{eq: switching_mechanism},
whose argument is allowed to depend linearly on the state $x_t$.
In this work, we assume the parameters
\rev[theta_trans_problem_formulation]{$\Theta \in \re^{n_x \times d}$}{$\Theta \in \re^{d \times n_x}$}, $A_i \in \re^{n_x \times n_x}$,
$B_i \in \re^{n_x \times n_u}$, $i \in \Xi$ to be known.
In practice, these can be estimated from data using \cite{wang2024em++}.

\begin{remark}
    Note that the argument of $\sigma$ in \cref{eq: switching_mechanism}
    may as well be an affine function of $x_t$, allowing to model exogenous switching 
    of $\xi_t$, for instance. However, we will
    not account for this case explicitly as it can simply be 
    modeled by augmenting the state $x_t$ with a constant 1.
\end{remark}

\begin{figure}[tb!]
    \centering
    \vspace*{1mm}
    \resizebox{0.8\linewidth}{!}{
    \usetikzlibrary{decorations.pathreplacing}
\colorlet{modecolor}{MidnightBlue}
\def\nodesize{10pt}
\begin{tikzpicture}
    {\footnotesize
    \tikzset{
        label/.style={draw=none, font=\scriptsize, inner sep=0, outer sep=0, node distance=20pt}, 
        toplabel/.style={draw=none, 
                     font=\scriptsize,
                     inner sep=0, outer sep=5pt,
                     node distance=13pt, anchor=south, color=gray,
                     fill=white},
        braces/.style={decorate,decoration={brace,amplitude=4pt},xshift=4pt, gray}, 
        braceslabel/.style={align=center, anchor=west, gray, font=\scriptsize, midway, inner sep=10pt},
        conn/.style={fill=white, fill opacity=0.8, text opacity=1,
                    inner sep=2pt, draw=none, text=gray},
        highlight/.style={
                        circle,
                        draw=modecolor,
                        line width=0.5mm,
                        minimum size=\nodesize,
                        fill=modecolor!20,
                        inner sep=0pt, 
                        anchor=center
                    }
    };

    \scenariotreecoordinates[%
        hspace=10em, vspace=1pt, nodesize=10pt, vspace=15pt,
        branching_factor=2, branching_level=2,horizon=2,skips=1
    ]{}

    \drawtimeline{$k$}{1.15}
    \foreach \t in {0,...,\hor}{
        \annotatestage{\t}{$\t$}{}
    }

    \fillsubtree{0}{\hor}{\name}
    \foreach \t in {0, ..., \hor}
    {
        \node[highlight] (node/\name/\t1) at (\name/\t1) {};
    }
    
    \newcounter{nodecnt}
    \setcounter{nodecnt}{0}
    \newcommand\indexnode[1]{
        \node (labelnode/#1) at (#1) {\thenodecnt};
        \stepcounter{nodecnt}
    }
    \newcommand\annotatenode[1]{
        \node[above of=labelnode/#1, label, gray, conn, node distance=12pt] {$(\nodevar{x}{\thenodecnt}, \nodevar{u}{\thenodecnt})$};
        \stepcounter{nodecnt}
    }
    \newcommand\annotateleaf[1]{
        \node[above of=labelnode/#1, label, gray, conn, node distance=12pt] {$\nodevar{x}{\thenodecnt}$};
        \stepcounter{nodecnt}
    }
    \doforeverynode{\indexnode}{0}{\hor}{}

    \setcounter{nodecnt}{0}
    \pgfmathparse{int(\hor-1)}
    \doforeverynode{\annotatenode}{0}{\pgfmathresult}{}
    \foreach \n in {1,...,4}{
        \annotateleaf{tree/\hor\n};
    };

    \newcounter{modenode}\setcounter{modenode}{1}
    \providecommand\labelconnection[4]{
        \foreach \lcltimectr in {#1,...,0}
        {
            \computenchildren{\lcltimectr} %
            \ifnum \nc>1
                \breakforeach 
            \fi
        }
        \path (node/tree/#1#2) --
              (node/tree/#3#4)
              node[conn,label,midway] (connect\themodenode) {
                \pgfmathparse{int(\nc - mod(#4, \nc))}  %
                $\nodevar{\md}{\themodenode}=\pgfmathresult$
            };
        \stepcounter{modenode}
    }
    \pgfmathparse{int(\hor-2)}
    \doforeveryconnection{\labelconnection}{0}{\pgfmathresult}
    \path (node/tree/11) --
          (node/tree/21)
          node[conn,label,midway] (connect\themodenode) {
          \pgfmathparse{int(\nc - mod(1, \nc))}  %
          $\nodevar{\md}{\themodenode}=\pgfmathresult$
    };

    \fill[opacity=0., modecolor, rounded corners] ($(connect1.north west) - (5pt,-5pt)$) coordinate(topleft) rectangle ($(connect2.south east) + (5pt,-5pt)$); 
        \node[text=modecolor, anchor=south, fill opacity=0.8, fill=white, text opacity=1, yshift=-55pt] at (topleft) {
            $\md_0=\left(\nodevar{\md}{1},\nodevar{\md}{2}\right) $
        };

    \node[text=modecolor, left of=node/\name/24, yshift=-5pt, xshift=-15pt] {$\nodes(N)=\{3, 4, 5, 6\}$};
    \node[text=modecolor, left of=node/\name/21, yshift= 5pt, xshift=-10pt] {$\scenanc(3) = \{0, 1\}$};

    }
\end{tikzpicture}
    }\caption{\small{A fully branched scenario tree \rev[specification_d]{of horizon $N=2$}{with $d=2$ and $N=2$}.}}
    \label{fig: scen-tree-illustration}
\end{figure}

Let $N$ be the prediction horizon.
Since $\Xi$ is a finite set,
the realization of the stochastic process $\{x_t, u_t\}_{t = 1}^N$ satisfying \eqref{eq: system} can be represented on a scenario tree \cite{pflug_MultistageStochasticOptimization_2014},
as illustrated in \cref{fig: scen-tree-illustration}.
We briefly introduce the scenario tree notation needed for this paper.
For detailed notation, we refer to \cite{wang2023interaction}.

A scenario tree captures all possible outcomes of a stochastic process over $N$ time steps.
We denote the nodes at step $k \in \N_{[0, N]}$ by $\nodes(k)$ and $\nodes\left(a,b\right) = \bigcup_{k=a}^{b} \nodes(k)$ for $a < b \in \N_{[0, N]}$.
Let $\ve{u} \dfn \{\nodevar{u}{\iota}\}_{\iota\in\nodes(0, N-1)}$,
and \rev{system}{} states $\ve{x} {\dfn} \{\nodevar{x}{\iota}\!\}_{\!\iota\in\nodes(0, N)}$.
The \rev{stochastic}{}dynamics \eqref{eq: system} can be represented as 
\[
    \nodevar{x}{\nextnode{\iota}} = A_{\nodevar{\xi}{\rev[fontsize]{\nextnode{\iota}}{\mathlarger{\nextnode{\iota}}}}}\nodevar{x}{\iota} + B_{\nodevar{\xi}{\rev{\nextnode{\iota}}{\mathlarger{\nextnode{\iota}}}}}\nodevar{u}{\iota},
\]
where $\nextnode{\iota}$ denotes the child node of $\iota$ corresponding to
the realization $\nodevar{\xi}{\nextnode{\iota}}$.
Each leaf node $s \in \nodes(N)$ corresponds to a unique scenario—a path from root to leaf representing one realization of the stochastic process.
For $s \in \nodes(N)$, 
we denote by $\scenanc(s)$ all nodes on this path excluding $s$ itself.
Let the convex functions $\ell(x, u): \re^{n_x} \times \re^{n_u} \to \re$
and $\ell_N(x): \re^{n_x} \to \re$ denote the stage costs and terminal cost respectively.
Then, the cost associated with a scenario $s$ is
\begin{equation}\label{eq: cost_scenario}
    L_{s}(\ve{x}, \ve{u}) = \ell_{N}(\nodevar{x}{s}) 
        + \textstyle{\sum_{\iota \in \scenanc(s)}} \ell(\nodevar{x}{\iota}, \nodevar{u}{\iota}).
\end{equation}

A conventional stochastic MPC formulation considers minimizing the expected cost in the optimal control problem (OCP), i.e., 
for a given initial state $x_t \in \mathbb{X}$,
\begin{alignat}{3}\label{eq: risk_neutral_ocp}
    \minimize_{(\ve{x}, \ve{u}) \in \mathbb{C}(x_t)} &\;&\mathcal{L}^{\mathrm{n}}(\ve{x}, \ve{u}) \dfn&
        \tlsum_{s\in \nodes(N)} p(s\midsc \ve{x}) L_{s}(\ve{x}, \ve{u}),
\end{alignat}
where the set $\mathbb{C}(x_t)$ is defined as
\begin{equation*}
    \begin{aligned}
        \mathbb{C}(x_t) \dfn \left\{
        \begin{array}{c}
            \hspace{-0.5em}
            (\ve{x}, \ve{u})
        \end{array}
        \middle|
        \resizebox{0.2\textwidth}{!}{$
        \begin{array}{l}
            \nodevar{x}{0}   = x_t, \\
            \nodevar{x}{\nextnode{\iota}} = A_{\nodevar{\xi}{\rev{\nextnode{\iota}}{\mathlarger{\nextnode{\iota}}}}}\nodevar{x}{\iota} + B_{\nodevar{\xi}{\rev{\nextnode{\iota}}{\mathlarger{\nextnode{\iota}}}}}\nodevar{u}{\iota}, \\
            \nodevar{u}{\iota} \in \mathbb{U}, \, \forall \iota \in \nodes(0, N-1), \\
            \nodevar{x}{\iota} \in \mathbb{X}, \, \forall \iota \in \nodes(0, N),
        \end{array}
        $}
        \right\}.
    \end{aligned}
\end{equation*}
The probability for a given scenario is defined as
\begin{equation}\label{eq: prob_scenario}
    p(s\midsc \ve{x}) 
    = \tlprod_{\iota \in \scenanc(s)} p(\nodevar{\xi}{\nextnode{\iota}}\midsc \ve{x})
    = \tlprod_{\iota \in \scenanc(s)} \sigma_{\nodevar{\xi}{\rev{\nextnode{\iota}}{\mathlarger{\nextnode{\iota}}}}}\big(
        \rev{\Theta^\top \nodevar{x}{\iota}}{\Theta \nodevar{x}{\iota} }
    \big).
\end{equation}
We denote by $P(\ve{x})  \in \Delta^{d^N}$ the vector whose elements represent the scenario probability 
$p(s \midsc \ve{x})$.
We refer to \eqref{eq: risk_neutral_ocp} as the \emph{risk-neutral formulation}.
Due to the state-dependent distribution $P(\ve{x})$,
problem \eqref{eq: risk_neutral_ocp} is nonconvex,
even when the loss function $L_s$ is convex,
making the problem hard to solve.

Risk-sensitive formulations use an exponential utility function \cite{jacobson1973optimal,whittle1989entropy} in the OCP, 
which is defined as 
\(
    \rho_\gamma(L, \rev{p}{P}) \dfn \tfrac{1}{\gamma}\ln \mathbb{E}_{P\rev{(\ve{x})}{}}\big[
    \exp(\gamma L )\big],
\)
where $\gamma \neq 0$ determines the degree of risk-sensitivity.
For $\gamma > 0$, it is also known as the entropic risk measure \cite[Example 6.20]{shapiro2021lectures}.
It will be convenient
to let $\gamma > 0$ and define the OCPs
\begin{alignat}{3}\label{eq: optimistic_ocp}
    \smashoperator{\minimize_{(\ve{x}, \ve{u}) \in \mathbb{C}(x_t)}} 
    \loss^{\mathrm{o}}\!(\ve{x},\! \ve{u})
    &\dfn
    {-}\!\tfrac{1}{\gamma}\ln\! 
    \E_{P(\ve{x})}
    \big[
        \exp \big({-}\gamma L(\ve{x},\! \ve{u}) \big)
    \big],\\
    \label{eq: pessimistic_ocp}
    \smashoperator{\minimize_{(\ve{x}, \ve{u}) \in \mathbb{C}(x_t)}} \loss^{\mathrm{p}}\!(\ve{x},\! \ve{u})
    &\dfn
    \phantom{-}\!\tfrac{1}{\gamma}\ln\! 
    \E_{P(\xx)}
        \big[
        \exp\big(\phantom{-}\gamma L(\ve{x},\! \ve{u}) \big)
    \big].
\end{alignat}
We refer to \eqref{eq: optimistic_ocp} as the \emph{optimistic formulation} and to \eqref{eq: pessimistic_ocp} as the \emph{pessimistic formulation}.

\subsection{Interpretation of the risk-sensitive formulation}\label{sec: risk_sensitive_interpretation}
\subsubsection{Regularized risk-neutral problem}\label{sec: connection_risk_neutral}
As presented in \cite[Remark 2]{tembine2013risk},
the risk-sensitive formulation can be viewed as a regularized risk-neutral problem,
in the form:
\begin{align*}
    \loss^\opti &\; = \E_{P(\ve{x})}(L) - \tfrac{\gamma}{2} \var_{P(\ve{x})}(L) + o(\gamma), \\
    \loss^\pess &\; = \E_{P(\ve{x})}(L) + \tfrac{\gamma}{2} \var_{P(\ve{x})}(L) + o(\gamma),
\end{align*}
when $\gamma$ is around 0.
The approximation is obtained by applying Taylor expansion.
For the optimistic case, higher cost variance across the scenario distribution $P(\ve{x})$ is preferred.
In contrast, the pessimistic case penalizes the cost variance across scenarios.

\subsubsection{Game-theoretical interpretation}

Risk-sensitive formulations \eqref{eq: optimistic_ocp} and \eqref{eq: pessimistic_ocp}
can be viewed as smoothed versions of a collaborative and a \rev{zero-sum}{competitive} game, respectively,
as we show in the following result.

\begin{lemma} \label{lem:limits}
	For any $x_t \in \mathbb{X}$, and $(\xu) \in \C(x_t)$, we have
	\[
    {\small 
		\underbar{L}(\xu)
		\leq
		\loss^\opti(\xu)
		\leq
    \E_{\rev{}{P(\xx)}}[L(\xu)]
		\leq
		\loss^\pess(\xu)
		\leq
		\bar{L}(\xu),
  }
	\]
	where $\underbar{L} \dfn \vecmin L$ and $\bar{L} \dfn \vecmax L$.
	Furthermore,
	\begin{enumerate}[label=(\roman*)]
		\item \label{item:o0} $\lim_{\gamma \to 0} \loss^\opti(\xu) = \E_{\rev{p}{P}(\xx)}[ L(\xu) ]$;
		\item \label{item:ooo} $\lim_{\gamma \to \infty} \loss^\opti(\xu) = \vecmin L(\xu)$;
		\item \label{item:p0} $\lim_{\gamma \to 0} \loss^\pess(\xu) = \E_{\rev{p}{P}(\xx)}[ L(\xu) ]$;
		\item \label{item:poo}$\lim_{\gamma \to \infty} \loss^\pess(\xu) = \vecmax L(\xu)$.
	\end{enumerate}
\end{lemma}

\begin{proof}
    Let $L \in \R^M$, and $p \in \interior\Delta^M$ be arbitrary vectors with $M \in \N_+ \setminus \{0\}$.

	The first and last inequality follow from the fact that
    $\gamma \underline{L} \leq \gamma L_i \leq \gamma \bar{L}$ and that $\exp$ and $\log$ are monotone
	increasing functions.
    The other inequalities follow directly from Jensen's inequality: 
	$\exp(\gamma \sum_{i=1}^M p_i L_i) \leq \sum_{i=1}^M p_i \exp(\gamma L_i)$.
	Composing with $\ln(\cdot) / \gamma$
	gives the third inequality.
	Applying the Jensen's inequality to $x \mapsto \exp(-\gamma x)$ 
	and composing with $\ln (\cdot)/ (-\gamma)$ gives the second inequality, 
	with the inequality direction flipped due to division by $-\gamma$.
	
	Applying l'H\^opital's rule to 
	\(
	\lim_{\gamma \to 0}\rho_{\gamma}(L, p)
	\), we obtain 
	\(\lim_{\gamma \to 0} \frac{\sum_{i=1}^M L_i p_i \exp (\gamma L_i)}{\sum_{i=1}^M p_i \exp(\gamma L_i)} = \sum_{i=1}^M L_i p_i\),
    which establishes \cref{item:o0,item:p0}.

	To establish \cref{item:poo}, let $\underline{p} \dfn \min \{ p_i : p_i > 0, 1 \leq i \leq M \} \in (0, 1]$. Note 
    that $\underline{p}$ must exist since $p_i$ are nonnegative and sum to 1.
    Using this, we have the inequalities
	$\tfrac{1}{\gamma} (\log \underline{p} + \lse(\gamma L)) \leq \rho_{\gamma}(L, p) \leq \frac{1}{\gamma}\lse(\gamma L)$, for $\gamma > 0$.
    Taking the limit $\gamma \to \infty$ on all sides of the inequality yields
    $\vecmax L \leq \lim_{\gamma \to \infty} \rho_{\gamma}(L, p) \leq \vecmax L.$
    Finally, \cref{item:ooo} is shown analogously. 
\end{proof}

From \cref{lem:limits}, it follows that as $\gamma$ increases, 
problems \eqref{eq: optimistic_ocp} and $\eqref{eq: pessimistic_ocp}$ 
more closely approximate the two-player games
\[ 
        \minimize_{\xu \in \C(x_t)} \min_{s \in \nodes(N)} L_s(\xu)\text{, }
         \minimize_{\xu \in \C(x_t)} \max_{s \in \nodes(N)} L_s(\xu),
\]
respectively, where in the former case, the controller 
takes the optimistic assumption that the uncertain environment will act cooperatively,
whereas in the latter, it assumes an adversarial environment.

\begin{remark}\label{rmk: convergence}
    Note that \cref{lem:limits} only provided point-wise convergence w.r.t. $\gamma$.
    Under some mild regularity assumptions, epigraphical 
    versions of these limits can also be shown, which provide a more complete justification 
    for this interpretation \cite[Ch. 7]{rockafellar2009variational}.
\end{remark}

\section{Sequential Convexification}\label{sec: mm_solver}
The problems \eqref{eq: optimistic_ocp} and \eqref{eq: pessimistic_ocp} are nonconvex problems.
To solve them efficiently, we propose a sequential convexification scheme based on the MM principle \cite{lange2016mm},
where we iteratively construct a convex 
surrogate function $\surrogate^m$ at the current iterate $(\ve{x}^m, \ve{u}^m) \in \mathbb{C}(x_t)$, satisfying
\begin{subequations}\label{eq: property_Q}
	\begin{align}
		\surrogate^m(\ve{x}, \ve{u})     \geq & \; \loss(\ve{x}, \ve{u}), \, \forall (\ve{x}, \ve{u}) \in \mathbb{C}(x_t),  \label{eq: Q_pos} \\
		\surrogate^m(\ve{x}^m, \ve{u}^m) =    & \; \loss(\ve{x}^m, \ve{u}^m) \label{eq: Q_equal_func}.
	\end{align}
\end{subequations}
which is minimized to yield the next iterate $(\ve{x}^{m+1}, \ve{u}^{m+1})$.

We show in the following proposition that a surrogate function satisfying \eqref{eq: Q_pos} for problem \eqref{eq: optimistic_ocp} can be obtained by applying Jensen's inequality.

\begin{proposition}\label{lem: surrogate_optimistic}
    Let $\Pi \in \Delta^{d^N}$ be an arbitrary distribution, and 
	\rev[def_kl_div]{}{$\kl{\Pi}{P} = \sum_{i} \Pi_i (\ln \Pi_i - \ln P_i)$ denote the Kullback-Leibler (KL) divergence, we}
	define
    \begin{equation}  \label{eq: def-Q-opti}
        \surrogate^{\mathrm{o}}\!(\ve{x}, \ve{u} \midsc \Pi) 
            \dfn \tfrac{1}{\gamma}\kl{\Pi}{P(\ve{x})} 
            {+} \Pi^\top\! L(\ve{x}, \ve{u}).
    \end{equation}
	For all $\ve{x}, \ve{u}$, we have
    $\loss^{\mathrm{o}}(\ve{x}, \ve{u}) \leq \surrogate^{\opti}(\ve{x}, \ve{u} \midsc \Pi)$.
\end{proposition}
\begin{proof}
	Recalling the definition \eqref{eq: optimistic_ocp}
	\rev{}{and noting that $x \mapsto -\ln(x)$ is convex},
	we expand 
	\begin{align*}
		\loss^\opti(\ve{x}, \ve{u}) 
		=&\; -\tfrac{1}{\gamma}\ln \Big[
				\tlsum_{s \in \nodes(N)} \frac{\Pi_{s}}{\Pi_{s}} p(s\midsc \ve{x}) \exp(- \gamma L_{s}(\ve{x}, \ve{u}))
			\Big] \\
		\leq &\; - \tfrac{1}{\gamma}\tlsum_{s \in \nodes(N)} \Pi_{s} \ln \big[ 
				\frac{p(s\midsc \ve{x})}{\Pi_{s}} \exp(- \gamma L_{s}(\ve{x}, \ve{u}))
			\big] \\
		= &\; - \tfrac{1}{\gamma} \tlsum_{s \in \nodes(N)} \Pi_{s} \ln \big[ 
			\frac{p(s\midsc \ve{x})}{\Pi_{s}}
			\big] + \smashoperator{\tlsum_{s \in \nodes(N)}} \Pi_{s} 
				L_{s}(\ve{x}, \ve{u}),
	\end{align*}		
    where the inequality follows the Jensen's inequality.
	The claim follows by definition of the KL divergence.
\end{proof}

Finding $\Pi$ satisfying \eqref{eq: Q_equal_func} is less straightforward.
Since $\surrogate^{\opti}(\cdot, \cdot\;; \Pi)$ upper bounds the 
cost for any $\Pi$, a reasonable estimate is to find the smallest such upper bound:
\begin{equation}
    \Pi^{\star} = 
        \textstyle{\argmin_{\Pi \in \Delta^{d^N}}} \surrogate(\ve{x}^m, \ve{u}^m \midsc \Pi). \label{eq: majorization_optimistic}
\end{equation}
Particularly,
this convex problem has a closed-form solution.
\begin{lemma}\label{lem: solution_optimistic_majorization}
	The solution of problem \eqref{eq: majorization_optimistic} is
	\begin{equation}\label{eq: sol_majorization_optimistic}
		\Pi^\star_{\rev{\ve{\xi}}{s}} = \tfrac{\exp\big(
			\ln p(s\midsc \ve{x}^m) - \gamma L_{s}(\ve{x}^m, \ve{u}^m)
		\big)}{
			\sum_{s'\in\nodes(N)} \exp\big(
				\ln p(s'\midsc \ve{x}^m) - \gamma L_{s'}(\ve{x}^m, \ve{u}^m)
			\big)
		}.
	\end{equation}
\end{lemma}
\begin{proof}
	Since $(\ve{x}^m, \ve{u}^m)$ are constants, we will
use $P$ and $L$ to denote $P(\ve{x}^m)$, $L(\ve{x}^m, \ve{u}^m)$, respectively.
Problem \eqref{eq: majorization_optimistic} is equivalent to
\begin{alignat}{3}
	\argmin_{\Pi} & \; & \tfrac{1}{\gamma} \kl{\Pi}{P} + & \; \Pi^\top L, \label{eq: majorization_equivalent} \\
	\stt          & \; & \Pi^\top\mathbf{1}_{d^N}  =     & \; 1, \quad
	\Pi_{s} \geq 0 \quad \forall s\in\nodes(N), \nonumber
\end{alignat}
As KL divergence is convex w.r.t. both arguments, the problem \eqref{eq: majorization_equivalent} is convex.
Define Lagrangian
\[
	\mathscr{L}(\Pi, \mu, \ve{\lambda}) \dfn
	\tfrac{1}{\gamma} \kl{\Pi}{P} + \Pi^\top L
	-  \ve{\lambda}^\top \Pi
	- \mu \big(1 - \Pi^\top\mathbf{1}_{d^N}\big).
\]
The KKT conditions of problem \eqref{eq: majorization_equivalent} yield:
\begin{subequations}
	\begin{align*}
		\nabla_{\Pi_{s}}\!\mathscr{L}(\Pi, \mu, \ve{\lambda})
		=  \tfrac{1}{\gamma}(\ln \Pi_{s} {+} 1 - \ln P_{s}) + L_{s} - \lambda_{s} + & \; \mu =  0, \\
		\lambda_{s}\Pi_{s} = 0,        \,
		\lambda_{s} 			  \geq 0,           \,
		\Pi_{s} 				  \geq 0,           \,
		\Pi^\top\mathbf{1}_{d^N} = 1, \,
		\forall s\in                                                                & \;\nodes(N).
	\end{align*}
\end{subequations}
The first equation is equivalent to
\begin{equation}\label{eq: sol_pi_intermediate}
	\Pi_{s} = \exp\big(
	-1 + \ln P_{s} + \gamma (\lambda_{s} - L_{s} - \mu)
	\big).
\end{equation}
Thus, $\Pi_{s} > 0$ for all $s\in\nodes(N)$ and $\lambda_{s} = 0$.
Because
\[
	\smashoperator{\tlsum_{s \in \nodes(N)}} \Pi_{s}
	= \exp(-\gamma \mu)\smashoperator{\tlsum_{s \in \nodes(N)}} \exp\big(
	-1 + \ln P_{s} - \gamma L_{s}
	\big) = 1,
\]

\vspace*{1mm}%
\noindent it follows that
\(
\mu = \tfrac{1}{\gamma} \ln \big(\smashoperator{\tlsum_{s\in\nodes(N)}} \exp(
-1 + \ln P_{s} - \gamma L_{s}
)\big).
\)
Plugging the expression into \eqref{eq: sol_pi_intermediate}, we obtain
\[
	\Pi_{s} = \tfrac{\exp\big(
		-1 + \ln P_{s} - \gamma L_{s}
		\big)}{\sum_{s'\in\nodes(N)} \exp\big(
		-1 + \ln P_{s'} - \gamma L_{s'}
		\big)
	},
\]
which is equivalent to \eqref{eq: sol_majorization_optimistic}.

\end{proof}

Plugging \eqref{eq: sol_majorization_optimistic} into the definition of $\surrogate^{\opti}$ \eqref{eq: def-Q-opti}, 
it is straightforward to verify that for this value of $\Pi$, $\surrogate^{\opti}$ satisfies \eqref{eq: Q_equal_func}.
To summarize, instead of solving \eqref{eq: optimistic_ocp} directly, we solve 
\rev{
\begin{alignat*}{3}
	(\ve{x}^{m+1}, \ve{u}^{m+1})\leftarrow &\;
	\textstyle{\argmin_{(\ve{x}, \ve{u})\in \mathbb{C}(x_t)}}\surrogate^{\mathrm{o}}(\ve{x}, \ve{u} \midsc \Pi^m)
\end{alignat*}
}{}
\usetagform{colorednumberequation}
\begin{equation}\label{eq: optimistic_surrogate}
	(\ve{x}^{m+1}, \ve{u}^{m+1})\leftarrow
	\textstyle{\argmin_{(\ve{x}, \ve{u})\in \mathbb{C}(x_t)}}\surrogate^{\mathrm{o}}(\ve{x}, \ve{u} \midsc \Pi^m)
\end{equation}
\usetagform{default}
at each iteration $m$, where the parameter $\Pi^m$ is computed via \eqref{eq: sol_majorization_optimistic} using the current iterate $(\ve{x}^m, \ve{u}^m)$.

\rev{Since the derivation in \cref{lem: surrogate_optimistic} relies on the convexity of function $x \mapsto -\ln(x)$,
this approach is not applicable for the pessimistic formulation.}{
	\cref{lem: surrogate_optimistic} relies on the convexity of function $x \mapsto -\ln(x)$ to derive the surrogate function,
	which is not applicable for the pessimistic formulation.
}
The following proposition provides an alternative
surrogate function satisfying \eqref{eq: property_Q} for problem \eqref{eq: pessimistic_ocp}.
\begin{proposition}\label{lem: surrogate_pessimistic}
	Let $\ve{\tilde{x}} = \{\nodevar{\tilde{x}}{\iota}\}_{\iota\in\nodes(0, N)}$,
	with $\nodevar{\tilde{x}}{\iota} \in \mathbb{X}$ chosen arbitrarily
    for all $\iota\in\nodes(0, N)$,
	and let
	\[
		\surrogate^{\pess}(\ve{x}, \ve{u} \midsc \ve{\tilde{x}})
		\dfn
        \tfrac{1}{\gamma}\lse \big[
			\hat{P} (\ve{x}\midsc \ve{\tilde{x}})
			+ \gamma  L(\ve{x}, \ve{u})\big],
	\]
	where the vector $\hat{P} (\ve{x}\midsc \ve{\tilde{x}}) \in \re^{d^{N}}$ has entries
	$\hat{P}_{s} (\ve{x}\midsc \ve{\tilde{x}}) {\dfn}
		\rho(\ve{\xi}_{s},\rev{\Theta^\top \ve{x}_{s}}{\Theta \ve{x}_{s}} \midsc \rev{\Theta^\top \ve{\tilde{x}}_{s}}{\Theta \ve{\tilde{x}}_{s}})$, 
        for 
        $s \in \nodes(N)$.
	Here,
	$\ve{\xi}_{s} \dfn \{\nodevar{\xi}{\nextnode{\iota}}\}_{\iota\in \scenanc(s)}$
	is the mode realization of scenario $s$,
	$\ve{x}_s \dfn \{\nodevar{x}{\iota}\}_{\iota\in \scenanc(s)}$
	represents the states along the scenario,
	and $\rho: \Xi^{N} \times \re^{d\times N} \to \re$ is defined as
	\begin{equation}\label{eq: rho}
		\rho (i,\, Z \midsc \tilde{Z})
		\dfn \tlsum_{t=0}^{N-1} \tilde{Z}_{i_t, t}
		- \lse(\tilde{Z}_t)
		- \sigma^\top(\tilde{Z}_t) (Z_t - \tilde{Z}_t).
	\end{equation}

    For all $\ve{x}, \ve{u}$ we have 
    \(
		\loss^\mathrm{p}(\ve{x}, \ve{u})
		\leq \surrogate^{\mathrm{p}}(\ve{x}, \ve{u} \midsc \ve{\tilde{x}})
    \).
    If, furthermore, $\ve{\tilde{x}} = \ve{x}$, this holds with equality.

\end{proposition}
\begin{proof}
	Function $\loss^{\mathrm{p}}$ defined in \eqref{eq: pessimistic_ocp}
	can be written as
	\begin{equation*}
		\begin{aligned}
			\loss^{\mathrm{p}}(\ve{x}, \ve{u})
			= & \;
			\tfrac{1}{\gamma}\ln \big[
				\tlsum_{s\in\nodes(N)} \exp \big(
				\ln p(s\midsc \ve{x}) + \gamma L_{s}(\ve{x}, \ve{u})
				\big)\big],
			\\
			= & \;
			\tfrac{1}{\gamma}\lse \big(
			\ln P(\ve{x}) + \gamma L(\ve{x}, \ve{u})
			\big).
		\end{aligned}
	\end{equation*}
	By definition \eqref{eq: prob_scenario}, we rewrite
	the log-probability
	\begin{align}\label{eq_prof: log_prob_scen}
		\ln p(s\midsc \ve{x})
		{=} \smashoperator{\tlsum_{\iota\in\scenanc(s)}} \ln \sigma_{\nodevar{\xi}{\rev{\nextnode{\iota}}{\mathlarger{\nextnode{\iota}}}}}\big(\rev{\Theta^\top\! \nodevar{x}{\iota}}{\Theta \nodevar{x}{\iota}} \big)
		{=} \smashoperator{\tlsum_{\iota\in\scenanc(s)}} \rev{
			\Theta_{\nodevar{\xi}{\nextnode{\iota}}}^\top \nodevar{x}{\iota}
		}{ 
			\Theta_{\nodevar{\xi}{\rev{\nextnode{\iota}}{\mathlarger{\nextnode{\iota}}}}}\nodevar{x}{\iota}
		} {-} \lse\big(\rev{\Theta^\top\! \nodevar{x}{\iota}}{\Theta \nodevar{x}{\iota}} \big),
	\end{align}
	where the second equality follows by definition of the softmax function $\sigma$.
	Since $\lse$ is convex and differentiable,
    it satisfies
	$\lse(x) \geq \lse(y) + \sigma^\top(y) (x - y)$.
	Applying this inequality to \eqref{eq_prof: log_prob_scen},
	it follows that
	\(
	\ln p(s\midsc \ve{x})
	\leq \rho(\ve{\xi}_{s},\, \rev[theta_trans_pessimistic]{\Theta^\top \ve{x}_{s}}{\Theta \ve{x}_{s}} \midsc \rev{\Theta^\top \ve{\tilde{x}}_{s}}{\Theta \ve{\tilde{x}}_{s}}).
	\)
	Consequently, $\ln P(\ve{x}) \leq \hat{P}(\ve{x}\midsc \ve{\tilde{x}})$ element-wise.
	Since $\lse$ is monontonically increasing,
	the inequality implies that
	\(
	\loss^\mathrm{p}(\ve{x}, \ve{u})
	\leq \tfrac{1}{\gamma}\lse \big(
	\hat{P}(\ve{x} \midsc \ve{\tilde{x}}) + \gamma L(\ve{x}, \ve{u})
	\big).
	\)
\end{proof}

To summarize, instead of solving \eqref{eq: pessimistic_ocp} directly,
we solve
\begin{alignat}{3}
	(\ve{x}^{m+1}, \ve{u}^{m+1})\leftarrow & \;
	\textstyle{\argmin_{(\ve{x}, \ve{u})\in \mathbb{C}(x_t)}}\surrogate^{\mathrm{p}}(\ve{x}, \ve{u} \midsc \ve{x}^m)
	\label{eq: pessimistic_surrogate}           \\
	\dfn                                   & \;
	\tfrac{1}{\gamma}\lse\big[
		\hat{P} (\ve{x}\midsc \ve{x}^m)
		+ \gamma  L(\ve{x}, \ve{u})\big]
	\nonumber ,
\end{alignat}
where each entry $\hat{P}_{s} (\ve{x}\midsc \ve{x}^m) \dfn
	\rho(\ve{\xi}_{s},\, \rev{\Theta^\top \ve{x}_{s}}{\Theta \ve{x}_{s}} \midsc \rev{\Theta^\top \ve{x}^m_{s}}{\Theta \ve{x}^m_{s}})$ with $\rho$ defined in \eqref{eq: rho}.

\section{Numerical Experiment}\label{sec: numerical_experiment}
We compare the performance of different formulations using a corridor scenario
where a robot and a human approach from opposite directions.
The robot is represented by discrete-time double integrators (in $x$ and $y$)
with state $x^{\mathrm{bot}}_k = \bmat{p_{x, k} & p_{y, k} & v_{x, k} & v_{y, k}}^\top$,
where $p_x, p_y, v_x, v_y$ denote the position and the velocity in the $x$ and $y$ directions, respectively.
The sample time is $\Delta t = \SI{0.1}{\second}$.
The robot controls its acceleration $u \in \re^2$ in both directions.
For simplicity, we assume the human (with state $x_k^{\mathrm{h}} = \bmat{p_{x, k}^{\mathrm{h}} & p_{y, k}^{\mathrm{h}}}$)
maintains a constant $x$-velocity $v^\mathrm{h}_x = \SI{ -0.8}{\meter / \second}$.
The $y$-velocity 
is controlled by 
$v^{h}_{y, k} = -0.3(p^{h}_{y, k} - y^{\mathrm{ref}}_{\xi})$,
to track a reference position $y^{\mathrm{ref}}_{\xi}$, 
taking values 
$y^{\mathrm{ref}}_1 = \SI{0}{\meter}$, 
$y^{\mathrm{ref}}_2 = \SI{-1}{\meter}$,
and
$y^{\mathrm{ref}}_3 = \SI{1}{\meter}$, depending on $\xi$.
The distribution of $\xi$ depends on the relative distance:
\begin{equation*}
    \xi \sim \sigma (\rev{\Theta^\top}{\Theta} \bsmat{
        p_x - p^{\mathrm{h}}_x \\
        p_y - p^{\mathrm{h}}_y \\
        1
    }),
    \quad 
    \text{with}
    \,
    \Theta = \rev[theta_trans_experiment]{\bsmat{
        -5 &  0 & -12.5 \\
        -1 &  1 & 0     \\
        -1 & -1 & 0     
    }}{
        \bsmat{
        -5 &         -1 &         -1 \\ 
         0 &          1 &         -1 \\        
         -12.5 &      0 &          0   
        }
    }.
\end{equation*}
We define the joint state $x_k \dfn \bmat{x_k^{\mathrm{bot}\, \top} & x_k^{\mathrm{h}\, \top} }^\top \in \re^6$.
The robot aims to move forward as fast as possible while avoiding collision either with the wall or the human.
To this end,
the robot tracks the reference 
$x_k^{\mathrm{ref}} = \bmat{p_{x, t} + \rev{}{k}v^{\max}_{x} \rev{(t + k}{}\Delta t\rev{)}{} & \!0\! & \!v^{\max}_{x}\! & \!0\!}$
using cost functions:
\[
    \begin{aligned}
    \ell^{\mathrm{track}}(x, u) 
    &= \tfrac{1}{2}\norm{x^{\mathrm{bot}} {-} x^{\mathrm{ref}}}^2_{Q} {+} \tfrac{1}{2}\norm{u}^2_{R},\\
    \text{and} \quad \ell_N^{\mathrm{track}}(x_k) 
    &= \tfrac{1}{2}\norm{x^{\mathrm{bot}} {-} x^{\mathrm{ref}}}^2_{\Qf},
    \end{aligned}
\]
and $Q {\dfn} \diag{50, \!50,\! 2, \!2}$, $R {\dfn} \diag{2, \!2}$, and $\Qf {\dfn} 5 Q$.
To avoid colliding with the human,
we introduce a penalty:
\begin{equation}\label{eq: collision_penalty}
    c(x) = \alpha \exp(- \beta \norm{\bsmat{
        p_x - p^{\mathrm{h}}_x \\
        p_y - p^{\mathrm{h}}_y
    }})
\end{equation}
with $\alpha = 500$ and $\beta = 5$.
Thus, the stage and the terminal cost are given by 
\(
    \ell(x_k, u_k) = \ell^{\mathrm{track}}(x_k, u_k) + c(x_k),
    \ell_N(x_N) = \ell_N^{\mathrm{track}}(x_N) + c(x_N).
\)
Finally, $\mathbb{X}$ and $\mathbb{U}$ are defined by
the box constraints $\bmat{-1 & -0.6}^\top \leq u \leq \bmat{1 & 0.6}^\top$,
$p_y \in [-1.5, 1.5]$, $v_x \in [0, 1.5]$, and $v_y \in [-1, 1]$.

\subsection{Upper Bound of Collision Penalty}\label{sec: collision_penalty}
The stage and terminal costs are nonconvex due to the collision penalty term \eqref{eq: collision_penalty}.
To apply our proposed formulation, 
we derive a convex upper bound for this collision penalty.

\begin{lemma}\label{lem: collision_penalty_upper_bound}
    Consider a function $f: \re^n \to \re$ defined as $f(x) = f_1(f_2(x))$, where:
    \begin{itemize}
        \item $f_1: \re \to \re$ is convex, monotonically decreasing,
        \item $f_2: \re^n \to \re$ is convex.
    \end{itemize}
    For any $\tilde{x} \in \re^n$ and $w \in \partial f_2(\tilde{x})$,
    \rev[subdifferential]{}{where $\partial f_2(\tilde{x})$ is the subdifferential of $f_2$ at $\tilde{x}$ \cite[\S 23]{rockafellar1997convex},}
    the function $\hat{f}(x \midsc w) = f_1(f_2(\tilde{x}) + w^\top (x - \tilde{x}))$
    is convex and satisfies:
    \begin{inlinelist}
    \item \label{item:ineq} \(
        \hat{f}(x \midsc w) \geq f(x)
        \) 
        for all $x \in \re^n$, 
    \item  \label{item:eq}\(
            \hat{f}(\tilde{x} \midsc w) = f(\tilde{x}).
        \)
    \end{inlinelist}
\end{lemma}
\begin{proof}
    \rev{}{Function} $\hat{f}$ is a composition of an affine mapping and a convex 
    function and is therefore convex \cite[Ex. 2.20(a)]{rockafellar2009variational}.
    Inequality \ref{item:ineq} follows directly from the subgradient inequality \cite[Prop. 8.12]{rockafellar2009variational},
    combined with the monotone decrease property of $f_1$.
    The equality \ref{item:eq} is readily verified by substituting 
    $\tilde{x}$ into the definitions of $f$ and $\hat{f}$.
\end{proof}

Several common collision penalty functions satisfy the conditions in \cref{lem: collision_penalty_upper_bound},
including:
\begin{itemize}
    \item $c(x) = \alpha \exp(- \beta \norm{x}^2_{\Sigma})$ with $\Sigma \succ 0$, such as in \cite{evens2022learning, nishimura2020risk}, 
        where $f_1(z) \dfn \alpha\exp(-\beta z)$, and $f_2(x) \dfn \norm{x}^2_{\Sigma}$;
    \item $c(x) = \alpha \exp(- \beta \norm{x})$ 
        where $f_1(z) \dfn \alpha \exp(- \beta z)$, and $f_2(x) \dfn \norm{x}$;
    \item $c(x) = 1 / (\alpha + \beta \norm{x})^p$ with $p > 0$, such as in \cite{wang2020game},
        where $f_1(z) \dfn 1 / (\alpha + \beta z)^p$, and $f_2(x) \dfn \norm{x}$;
\end{itemize}
where $\alpha, \beta > 0$. 
In our experiment,
we select the form $c(x) = \alpha \exp(- \beta \norm{x})$ 
since it \rev{is approximated quite well by}{exhibits small approximation gap with} the upper bound in \cref{lem: collision_penalty_upper_bound} \rev{}{across the domain}.
To apply our formulation, we first majorize the loss function $\loss^{\mathrm{o}}$ or $\loss^{\mathrm{p}}$
using \cref{lem: surrogate_optimistic} or \cref{lem: surrogate_pessimistic}, respectively.
Then, we replace the collision penalty with its convex upper bound by applying \cref{lem: collision_penalty_upper_bound},
using the previous iterate as $\tilde{x}$.

\subsection{\rev{Simulation with Different \texorpdfstring{$\gamma$}{gamma}}{Sensitivity of \texorpdfstring{$\gamma$}{gamma}}}\label{sec: compare_gamma}
\begin{figure}[tb]
    \centering
    \vspace{1mm}
    \includegraphics[width=0.45\textwidth]{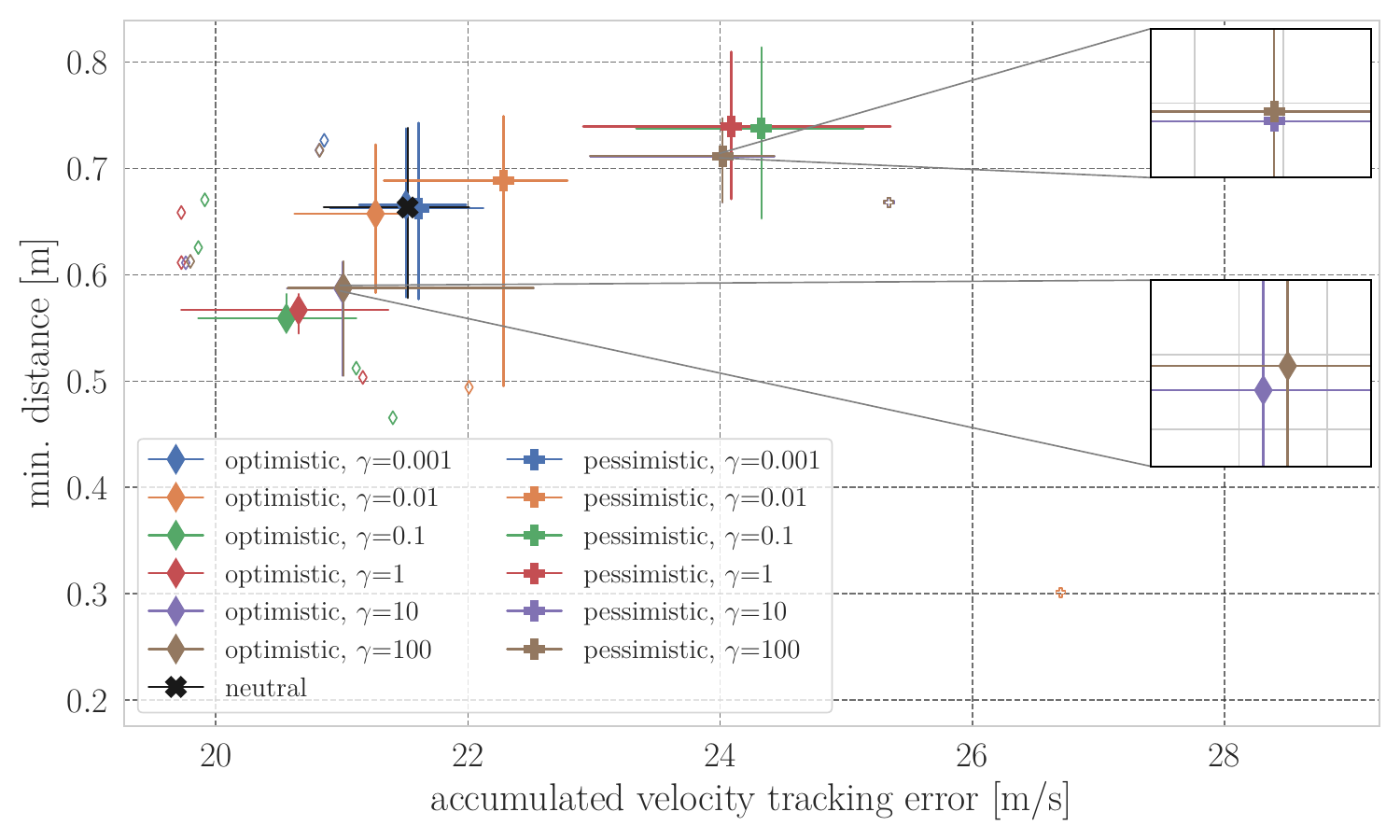}
    \caption{\small{Performance comparison of \rev{risk-sensitive formulation}{different formulations} with $\gamma = 10^{i}, i \in \N_{[-3, 2]}$. 
    The \rev{diamonds and plus signs}{solid markers} represent the median.
    The lines extend from $Q_1 - \mathrm{IQR}$ to $Q_3 + \mathrm{IQR}$,
    where $\mathrm{IQR}$ is the difference between the first and the third quartile ($Q_1$ and $Q_3$).
    Data points beyond this range are plotted individually as outliers. 
    Same color indicates the same value of $\gamma$.}}
    \label{fig: comparison_corridor}
\end{figure}
\rev[choice_gamma]{}{
    To test the impact of different $\gamma$ values,
    we set $\gamma = 10^{i}$, $i \in \N_{[-3, 2]}$.
    This range spans from small values that approximate the risk-neutral case to larger values that approach joint minimization and minimax problems.
}
\rev{}{
    To compare with the risk-neutral case,
we apply the same termination criterion measuring the optimality error \cite[Eq. (6)]{wachter2006implementation} for the proposed method,
with the tolerance $\epsilon_{\mathrm{tol}} = 0.003$.
}
To reduce the computational complexity,
we branch the scenario tree for first \rev{5}{$N_b=5$} steps for prediction horizon $N=20$. 
The final selected mode is maintained for the last 15 steps.
The robot is initialized at $x_0^{\mathrm{bot}} = \bmat{-3 & 0 & 0 & 0}^\top$
while the human position is randomly initialized.
Each experiment is repeated 10 times to mitigate randomness effects.
The closed-loop performance is evaluated by
the accumulated velocity tracking error
and the minimal distance between the robot and the human.
The result is summarized in \cref{fig: comparison_corridor}.

As shown in \cref{fig: comparison_corridor}, 
the optimistic formulation yields reduced velocity tracking error but smaller minimal distance than its pessimistic counterparts.
Both formulations exhibit similar behavior as they converge towards the risk-neutral formulation ($\gamma \to 0$).
For the optimistic formulation, 
initial increases in $\gamma$ simultaneously decreases velocity tracking error and minimal distance,
\rev{though}{as the optimistic formulation accounts for cooperative behavior that jointly minimizes the loss. Though}
at large $\gamma$ values, 
tracking error deteriorates 
due to the finite-infinite prediction horizon discrepancies.
Conversely, 
the pessimistic formulation initially exhibits increased tracking error and distance with rising $\gamma$, 
\rev{before demonstrating improved tracking at high $\gamma$ values—a result of}{
    as it accounts for unfavorable behavior that maximizes the loss.
    At high $\gamma$ values, it demonstrates improved tracking as a result of}
    the discrepancy between the finite-horizon worst-case approximation and the underlying infinite-horizon problem.
\rev{An example simulation is shown in \cref{fig: example_simulation}.
We observe that the trajectory controlled by the optimistic formulation takes a shorter detour around the human than the pessimistic formulation, 
resulting in lower velocity tracking error but closer distance.
\begin{figure}[tb]
    \centering
    \vspace{1mm}
    \includegraphics[width=0.4\textwidth]{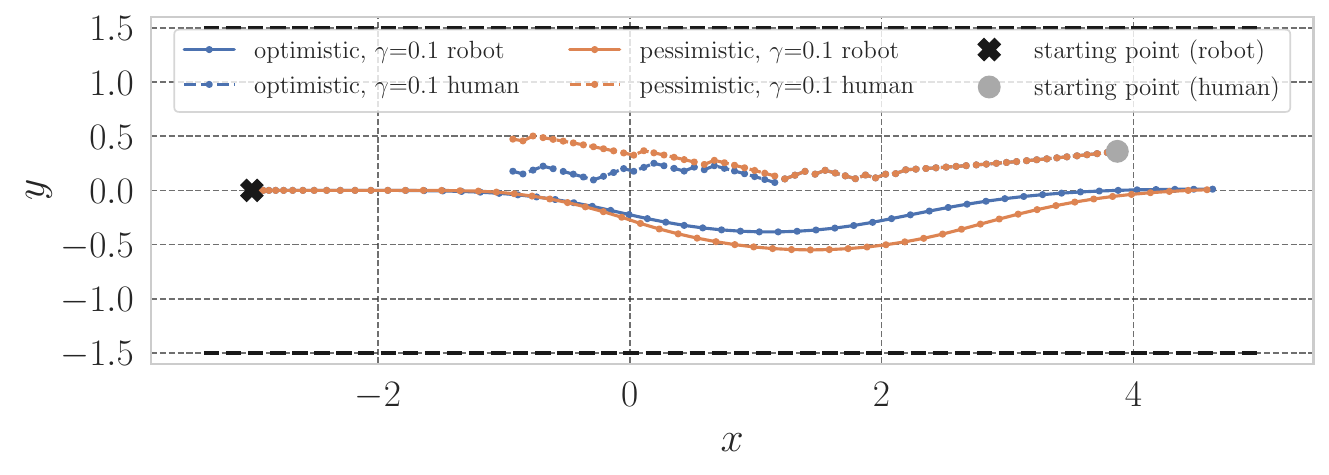}
    \caption{\small{Closed-loop trajectory of one simulation}}
    \label{fig: example_simulation}
\end{figure}
It is noteworthy that across all simulations,
the solver converges within 10 iterations using termination criterion 
$\loss^i(\ve{x}^{m}, \ve{u}^{m}) - \loss^i(\ve{x}^{m+1}, \ve{u}^{m+1}) \leq 1$ 
for all time steps, where $i \in \{\opti, \pess\}$.
Despite this relatively loose termination condition, 
the simulation results confirm that the controller maintains effective performance.
}{}

\subsection{Runtime Analysis}\label{sec: runtime_analysis}

As suggested in \cref{sec: connection_risk_neutral} and in \cref{sec: compare_gamma},
the risk-sensitive formulation approximates the risk-neutral formulation with a small $\gamma$.
We hence compare the runtime performance of the risk-sensitive formulation with $\gamma = 10^{-3}$ and the risk-neutral formulation
to demonstrate its computational advantage.
We implement the convex problem of the proposed MM algorithm using \texttt{CVXPY} \cite{diamond2016cvxpy},
and solve it via \texttt{MOSEK} \cite{mosek}.
The risk-neutral problem is solved using \texttt{IPOPT} \cite{wachter2006implementation}.
For problems \eqref{eq: risk_neutral_ocp}, \eqref{eq: optimistic_ocp}, and \eqref{eq: pessimistic_ocp} 
we select the initial state $x_t = \bmat{-2.5 & 0 & 1 & 0 & 1 & 0.2}^{\rev{}{\top}}$ to ensure constraints are activated within the horizon.
We branch the scenario tree for the first \rev{5}{$N_b=5$} steps with varying prediction horizons.
Using the same set of \rev{20}{10} random initial guesses for all algorithms,
the experiment is conducted on an Intel Core i7-11700@\SI{2.50}{G\hertz} machine.

To conduct a fair comparison,
we apply the same termination criterion measuring the optimality error \cite[Eq. (6)]{wachter2006implementation} for all algorithms
with the tolerance $\epsilon_{\mathrm{tol}} = 0.003$. 
\rev{The mean of the runtime is summarized in \cref{fig: runtime_comparison}.}{The optimal loss value and the mean of the runtime are summarized in \cref{tab: runtime_comparison}.}
\rev{The}{From \cref{tab: runtime_comparison}, we observe that each solver converges to a solution of similar quality measured by the loss value.
However, the} 
risk-sensitive formulations demonstrate a significant improvement in runtime performance compared to the risk-neutral formulation.
\rev{For risk-sensitive formulations,
we explicitly report the inner solver \texttt{MOSEK} runtime and termination criterion evaluation due to \texttt{CVXPY} compilation overhead.}{}

\rev{
\begin{figure}[tb]
    \centering
    \includegraphics[width=0.8\linewidth]{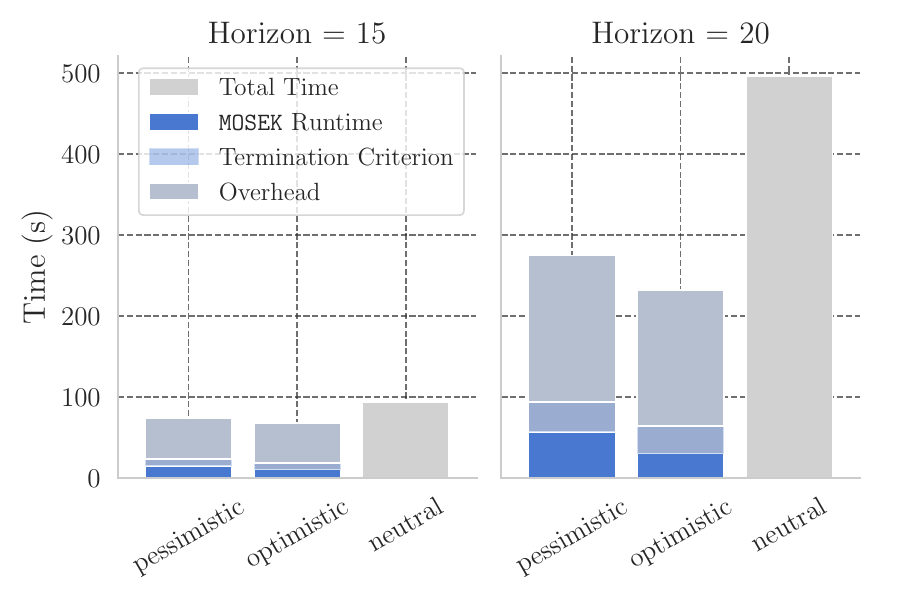}
    \caption{\small{Runtime comparison using termination criterion \cite[Eq. (6)]{wachter2006implementation}
        with tolerance $\epsilon_{\mathrm{tol}} = 0.003$.
        Risk-neutral problem is solved by \texttt{IPOPT}.
        The surrogate problems in MM algorithm are solved by \texttt{MOSEK}.
        We explicitly report the runtime of \texttt{MOSEK} and termination criterion evaluation due to the overhead of \texttt{CVXPY} compilation.
    }}
    \label{fig: runtime_comparison}
\end{figure}
}{}

\begin{table}[bt]
    \centering
    \vspace*{1mm}
    \caption{\rev{}{\small{Runtime comparison using termination criterion \cite[Eq. (6)]{wachter2006implementation}
    with $\epsilon_{\mathrm{tol}} = 0.003$.
    Risk-neutral problem is solved by \texttt{IPOPT}.
    The surrogate problems in MM algorithm are solved by \texttt{MOSEK}.}}}
    \label{tab: runtime_comparison}
    \rev{}{
    \resizebox{0.48\textwidth}{!}{
    \begin{tabular}{r rr rr}
        \toprule 
                      & \multicolumn{2}{c}{$N = 15$}                                    & \multicolumn{2}{c}{$N = 20$}\\
                      \cmidrule(lr){2-3}\cmidrule(lr){4-5}
                      & $\loss$         & runtime [\SI{}{\second}]  & $\loss$               & runtime [\SI{}{\second}] \\
        \midrule
        risk-neutral  & $85.3\pm1.1$    & $128.2\pm15.2$            & $187.4\pm3.8$         & $285.0\pm69.5$\\
        optimistic    & $83.8\pm0.0$    & $11.6\pm0.6$              & $183.6\pm2.9$         & $33.8\pm3.8$\\
        pessimistic   & $84.2\pm0.0$    & $17.1\pm0.8$              & $189.8\pm4.3$         & $60.1\pm4.9$\\
        \bottomrule
    \end{tabular}
    }}
\end{table}

\begin{table}[bt]
    \centering
    \vspace*{1mm}
    \caption{\rev[runtime_closed_loop]{}{Performance comparison of optimistic formulation using different scenario tree configurations.}}
    \label{tab: runtime_closed_loop}
    \rev{}{
    \resizebox{0.48\textwidth}{!}{
    \begin{threeparttable}
    \begin{tabular}{r rr rr}
        \toprule 
                                  & \multicolumn{2}{c}{Runtime per time step} [\SI{}{\second}] & min. dist. ($\uparrow$) [\SI{}{\meter}] & AVTE\tnote{\textdagger} $\,$($\downarrow$) [\SI{}{\meter / \second}]\\
                                  \cmidrule(lr){2-3}
                                  & mean            & max                  &  mean $\pm$ std.        & mean $\pm$ std. \\
        \midrule
        $N=15$, $N_b=2$\tnote{*}  & $0.091$         & $0.241$              & $0.532\pm0.039$         & $20.280\pm0.436$\\
        $N=20$, $N_b=5$           & $8.773$         & $44.342$             & $0.565\pm0.053$         & $20.544\pm0.467$\\
        \bottomrule
    \end{tabular}
    \begin{tablenotes}
        \item [\textdagger] Accumulated velocity tracking error
        \item [*] Run 1 MM iteration
    \end{tablenotes}
    \end{threeparttable}
    }}
\end{table}

\rev{}{
    The computational complexity grows exponentially w.r.t. the branching horizon $N_b$.
    In practice,
    different heuristics can be applied to reduce the runtime,
    such as using a loose termination criterion,
    or reducing the branching horizon $N_b$.
    \cref{tab: runtime_closed_loop} shows an example with 10 simulations,
    where these heuristics effectively reduce the runtime while maintaining similar closed-loop performance.
    However, this requires application-specific tuning.
    We leave the development of a general approach to reduce the complexity for future work.
}

\section{Conclusion}\label{sec: conclusion}
We propose a risk-sensitive MPC formulation for \rev{mixture models with decision-dependent distributions.}{an MoE model}
\rev{The proposed formulation}{that} enables the application of the MM principle\rev{}{.} 
\rev{to derive}{This approach derives} a convex upper bound of the nonconvex \rev{optimization}{optimal control} problem\rev{.}{}
\rev{This same principle can be applied to handle the}{with a} nonconvex collision penalty\rev{term}{},
allowing us to solve the original nonconvex \rev{optimal control}{}problem through a sequence of convex surrogate problems.
A numerical experiment validates the effectiveness and the runtime benefits of the proposed approach,
though the solver currently faces computational challenges due to the exponential growth in complexity from scenario tree branching. 
Future work will focus on developing an efficient convex inner solver and scenario reduction techniques to address this computational burden.
Moreover, since the \rev[sensitivity_conclusion]{system}{closed-loop} behavior is highly sensitive to the risk-sensitivity parameter,
we plan to investigate on \rev{adaptively tuning the value of the risk-sensitivity parameter}{adaptive parameter tuning} based on the specific control scenario.

\addtolength{\textheight}{-12cm}   %

\addcontentsline{toc}{section}{References}
\bibliographystyle{IEEEtran}
\bibliography{IEEEabrv,reference}

\end{document}